\documentclass[11pt,reqno]{amsart}

\usepackage{amsthm}               
\usepackage{amsmath}              
\usepackage{amsfonts}             
\usepackage{amssymb}    
\usepackage{enumerate} 
\usepackage{pstricks}      
\usepackage{multido} 
\usepackage{url}
\usepackage[margin=1in]{geometry} 

\theoremstyle{plain}
\newtheorem{thm}{Theorem}

\theoremstyle{definition}

\newtheorem{cor}[thm]{Corollary}

\newcommand{\pq}[2]{\genfrac{[}{]}{0pt}{}{#1}{#2}}
\newcommand{\B}{\mathcal{B}}
\newcommand{\T}{\mathcal{T}}
\newcommand{\V}{\mathcal{V}}
\newcommand{\U}{\mathcal{U}}
\newcommand{\Perm}{\mbox{Perm}}
\newcommand{\Part}{\mbox{Part}}
\newcommand{\wt}{\mbox{wt}}
\newcommand{\Det}{\mbox{\rm{det}}}

\begin{document}
	\author[R.B.~Corcino]{Roberto B.~Corcino}
	\address{Department of Mathematics, Mindanao State University, 7900 Marawi City, Philippines}
	\email[R.B.~Corcino]{rcorcino@yahoo.com}
	
	\author[K.J.M.~Gonzales]{Ken Joffaniel M.~Gonzales}
	\author[M.J.C.~Loquias]{Manuel Joseph C.~Loquias}
	\author[E.L.~Tan]{Evelyn L.~Tan}
	\address{Institute of Mathematics, University of the Philippines Diliman, 1101 Quezon City, Philippines}
	\email[K.J.M.~Gonzales]{kmgonzales@upd.edu.ph}
	\email[M.J.C.~Loquias]{mjcloquias@math.upd.edu.ph}
	\email[E.L.~Tan]{tan@math.upd.edu.ph}
	
	\title{Dually Weighted Stirling-type Sequences}

	\begin{abstract}
		We introduce a generalization of the Stirling numbers via symmetric functions involving two weight functions. The resulting extension unifies 
		previously known Stirling-type sequences with known symmetric function forms, as well as other sequences such as the $p,q$-binomial coefficients. 
		Recurrence relations, generating functions, orthogonality relations, convolution formulas, and determinants of certain matrices involving the 
		obtained sequences are derived. We also give combinatorial interpretations of certain cases in terms of colored partitions and permutations.
	\end{abstract}



	\date{\today}

	\maketitle

	\section{Introduction}
		Let $\mathbb N:={\{0,1,2,\ldots\}}$. For $n,k\in \mathbb N$ such that $n\geq k$, $q$-analogues of $n$, $n!$ and $\binom{n}{k}$, respectively are defined by 
		$[n]_q=1+q+q^2+\cdots+q^{n-1}$, with $[0]_q=0$, $[n]_q!=[n]_q[n-1]_q\cdots[1]_q$, and $[0]_q!=1$, and $\pq{n}{k}_{q}=\frac{[n]_q!}{[k]_q![n-k]_q!}$.  Also, 
		corresponding $p,q$-analogues have been defined as $[n]_{pq}=p^{n-1}+p^{n-2}q+p^{n-3}q^2+\cdots+q^{n-1}$, with $[0]_{pq}=0$, 
		$[n]_{pq}!=[n]_{pq}[n-1]_{pq}\cdots[1]_{pq}$, and $[0]_{pq}!=1$, and $\pq{n}{k}_{pq}=\frac{[n]_{pq}!}{[k]_{pq}![n-k]_{pq}!}$. We denote by $c(n,k)$ 
		and $S(n,k)$ the Stirling number of the first kind and Stirling number of the second kind, respectively, and their $p,q$-analogues by $c_{p,q}[n,k]$ and 
		$S_{p,q}[n,k]$. A number of authors have shown that these sequences have analogous symmetric function forms (see for example 
		\cite{Kon,Med,Med1}). Specifically, if $e_t(x_0,x_1,\ldots,x_r)$ denote the $t$-th elementary symmetric function and $h_t(x_0,x_1,\ldots,x_r)$ 
		the homogeneous symmetric function on the set ${\{x_0,x_1,\ldots,x_r\}}$, then
		\begin{align*}
			\binom{n}{k} & =e_{n-k}(1,1,\ldots,1)=h_{n-k}(1,1,\ldots,1)\\
			\pq{n}{k}_{q} &=h_{n-k}(1,q,q^2,\ldots,q^k)\\
			c(n,k) &= e_{n-k}(0,1,2,\ldots,n-1) \\
			c_{p,q}[n,k] &= e_{n-k}([0]_{p,q},[1]_{p,q},[2]_{p,q},\ldots,[n-1]_{p,q})\\
			S(n,k) &= h_{n-k}(0,1,2,\ldots,k) \\
			S_{p,q}[n,k] &= h_{n-k}([0]_{p,q},[1]_{p,q},[2]_{p,q},\ldots,[k]_{p,q})\,.
		\end{align*}

		Using $[n]_{pq}=p^{n-1}[n]_{q/p}$, and hence, $[n]_{pq}!=p^{\binom{n}{2}}[n]_{q/p}!$, we can show that $\pq{n}{k}_{p,q} = p^{k(n-k)}\pq{n}{k}_{q/p,1}$. Applying 
		the symmetric function expression for the $q$-binomial coefficients yields
		\begin{equation}
			\pq{n}{k}_{p,q} = h_{n-k}(p^k,p^{k-1}q,p^{k-2}q^2,\ldots,q^k)\,.\label{pghom}
		\end{equation}

		Indeed, Medicis and Leroux \cite{Med} have shown that it is efficient to use symmetric functions in unifying the treatment of many Stirling-type sequences.  In particular, the 
		$\U$-Stirling numbers given below reduce to many previously known generalizations of Stirling numbers for specific choices of the weight function $w$:	
		\begin{align*}
			c^\U(n,k)&=e_{n-k}(w_0,w_1,\ldots,w_{n-1})\\
			S^\U(n,k)&=h_{n-k}(w_0,w_1,\ldots,w_{k})
		\end{align*}
		Nonetheless, there is a continued interest in other types of Stirling numbers.  We mention here two recent generalizations of Stirling numbers that are also special cases of 
		$\U$-Stirling numbers:  Andrews and Littlejohn's Legendre Stirling numbers \cite{And}, obtained when $w_i=i(i+1)\cdots(i+m)$, and Miceli's poly-Stirling numbers \cite{Mic},
		obtained by letting $w_i=p(i)$, where $p$ is a polynomial in $i$ with coefficients from $\mathbb N$.  Moreover, it has been shown~\cite{Lun} that the $p,q$-binomial coefficients
		describe the magnetization distributions of the Ising model.
		
		In this paper, we introduce a generalization of the Stirling numbers which we shall call $\V$-Stirling numbers.  They are inspired by the symmetric function form of the 
		$p,q$-binomial coefficients in \eqref{pghom}. Sections 3--6 deal with the recurrence relations, generating functions, orthogonality relations, and convolution formulas of these
		numbers.  We also provide combinatorial interpretations of these numbers in Section 7 using certain colored permutations and partitions.

	\section{$\V$-Stirling Numbers}
		Let $\V=(v,w)$, where $v$ and $w$ are weight functions from $\mathbb Z$ to a commutative ring $K$ with unity, and let $\alpha,\beta\in\mathbb Z$. We define the \emph{$\V$-Stirling
		numbers of the first kind and second kind}, respectively, as
		\begin{align} 
			c^\V_{\alpha,\beta}[n,k] &:= e_{n-k}(v_{\alpha+n-1}w_{\beta},v_{\alpha+n-2}w_{\beta+1},\ldots,v_{\alpha}w_{\beta+n-1})\label{csym}\\
			S^\V_{\alpha,\beta}[n,k] &:= h_{n-k} (v_{\alpha+k}w_{\beta},v_{\alpha+k-1}w_{\beta+1},\ldots,v_{\alpha}w_{\beta+k})\label{Ssym}\,.
		\end{align}
		for $n,k\in\mathbb N$ with $k\leq n$. If $n,k<0$, we set both $c^\V_{\alpha,\beta}[n,k]$ and $S^\V_{\alpha,\beta}[n,k]$ to be $0$.  Note that the parameters $\alpha$ and $\beta$
		do not give more generality than the weight functions themselves.

		The $\V$-Stirling numbers reduce to the $\U$-Stirling numbers when $\V=(1,w)$ and $\alpha=\beta=0$.  We obtain the $p,q$-binomial coefficients when $v(i)=p^i$ and $w(i)=q^i$:
		\begin{align*}
			S^\V_{\alpha,\beta}[n,k]&=p^{\alpha(n-k)} q^{\beta(n-k)} \pq{n}{k}_{p,q}\\
			c^\V_{\alpha,\beta}[n,k]&=p^{\alpha(n-k)+\binom{n-k}{2}}q^{\beta(n-k)+\binom{n-k}{2}}\pq{n}{k}_{p,q}\,.
		\end{align*}
		In particular, $S^\V_{0,0}[n,k]=\pq{n}{k}_{p,q}$.  
		
		Barry introduced the sequence A080251 in OEIS \cite{Bar}, which we refer to as $b$-Stirling numbers of the second kind, via the following generating function 
		\begin{equation} 
			\sum_{n\geq k} S_{b}(n,k) x^n=\frac{x^k}{(1-T_{k-2,0}x)(1-T_{k-2,1}x)\cdots(1-T_{k-2,k}x)} \label{aah} 
		\end{equation} where 
		\[T_{k,j}=\left\{\begin{aligned} 
			&\big\lfloor \tfrac{j+2}{2}\big\rfloor \big(k-j+ \big\lfloor \tfrac{j+3}{2}\big\rfloor \big), &&\mbox{if $k\geq j$}\\
			&0, &&\mbox{if $k<j$}. 
		\end{aligned}\right.\]
		We define the $b$-Stirling numbers of the first kind as
		\begin{equation}
			\sum_{k\geq 0} c_{b}(n,k) x^k = (x+T_{n-3,0})(x+T_{n-3,1})\cdots(x+T_{n-3,n-1}). \label{ooh}
		\end{equation}
		Observe that $\big\lfloor \frac{j+2}{2}\big\rfloor +\big(k-2-j+ \big\lfloor\frac{j+3}{2}\big\rfloor\big)=k$.  Hence, $\{T_{k-2,j}\mid 0\leq j\leq k\}$ consists of all 
		products $ab$ where $1\leq a,b\leq k, a+b=k$.  This means that 
		\[(1-T_{k-2,0}x)(1-T_{k-2,1}x)\cdots(1-T_{k-2,k}x)=(1-v_kw_0x)(1-v_{k-1}w_1x)\cdots (1-v_{0}w_kx)\]
		where $v_i=w_i=i$.  Setting $\V=(i,i)$, we then have $S_b(n,k) = S^\V_{0,0}[n,k]$.  Similar computations give us $c_b(n,k)= c^\V_{0,0}[n,k]$.

		Another specific case is the $\zeta$-analogue of Stirling numbers introduced in \cite{Dzi}.  These are also $\V$-Stirling numbers that are obtained by letting 
		$\V=(\zeta^i,w_{i-1})$.

		We now interpret $\mathcal{V}$-Stirling numbers via arrays. Define $T_{\alpha,\beta}[r,s]$ to be the set of $2 \times s$ arrays satisfying the following four conditions:
		\begin{enumerate}[(a)]
			\item The entries on the first row are from the set ${\{\alpha,\alpha+1,\ldots,\alpha+r\}}$.
			\item The entries on the first row are nonincreasing from left to right.
			\item The entries on the second row are from the set ${\{\beta,\beta+1,\ldots,\beta+r\}}$.
			\item The sum of the entries in each column is $\alpha+\beta+r$.
		\end{enumerate}
		Also, we denote by $Td_{\alpha,\beta}[r,s]$ the subset of $T_{\alpha,\beta}[r,s]$ consisting of arrays whose first row entries are distinct.  The elements of 
		$Td_{\alpha,\beta}[r,s]$ and $T_{\alpha,\beta}[r,s]$ will be referred to as \emph{$\B$-tableaux}.  Conditions (b) and (d) above imply that the entries on the second row of a 
		$\B$-tableau are nondecreasing.  If $r$ or $s$ is negative, and if $r<s-1$, then $Td_{\alpha,\beta}[r,s]$ and $T_{\alpha,\beta}[r,s]$ are both empty.  Note that we allow a 
		$\B$-tableau to be a $2 \times 0$ array.
		
		We say that two $\B$-tableaux are \emph{compatible} if the constant column sum of both $\B$-tableaux are equal. Given two compatible $\B$-tableaux $\phi$ and
		$\phi'$, denote by $\phi\boxtimes\phi'$ the $\B$-tableau obtained by juxtaposing $\phi$ and $\phi'$, and then rearranging the columns so that entries on the
		first row are nonincreasing from left to right. Let $\T$ and $\T'$ be two sets of $\B$-tableaux such that every $\B$-tableau in $\T$ is compatible with
		every $\B$-tableau in $\T'$. Define $\T\boxtimes\T'$ to be the multiset
		\[\T\boxtimes\T'= {\{\phi\boxtimes\phi'\mid \phi\in\T,\phi'\in\T'\}}\]
		with $\varnothing\boxtimes\T=\T\boxtimes\varnothing= \varnothing$.  Note that $\T\boxtimes\T'=\T'\boxtimes\T$.  Moreover, if $\phi$ has zero columns, then
		$\phi\boxtimes\phi'=\phi'$.

		Let $\phi$ be a $\B$-tableau with $s$ columns. For a pair $\V=(v,w)$ of weight functions, we define the \emph{weight of $\phi$ with respect to $\V$} (or the \emph{$\V$-weight of 
		$\phi$}) by
		\[\V_{\wt}(\phi)= \prod_{j=0}^s v(\phi_{1,j}) w(\phi_{2,j})\]
		where $\phi_{i,j}$ is the $j$th column entry of the $i$th row of~$\phi$.  If a $\B$-tableau $\zeta$ has zero columns, then $\V_{\wt}(\zeta):=1$. Note that for any two compatible 
		$\B$-tableaux $\phi$ and $\phi'$, $\V_{\wt}(\phi\boxtimes\phi')=\V_{\wt}(\phi)\V_{\wt}(\phi')$.  

		From the above definitions, the $\V$-Stirling numbers may now be written as follows:
		\begin{align}
		c^{\V}_{\alpha,\beta}[n,k] &=\textstyle\sum_{\phi\in Td_{\alpha,\beta}[n-1,n-k]} \V_{\wt}(\phi) \label{cvwt}\\
		S^{\V}_{\alpha,\beta}[n,k] &=\textstyle\sum_{\phi\in T_{\alpha,\beta}[k,n-k]} \V_{\wt}(\phi)\label{svwt}
		\end{align}

		A bijection $\tau$ from $Td_{\alpha,\beta}[n-1,n-k]$ to $T_{\alpha,\beta}[k,n-k]$ may be defined as follows: if $\phi\in Td_{\alpha,\beta}[n-1,n-k]$ then $\tau(\phi)=\phi'$ where
		$\phi'_{1,j}=\phi_{1,j}-(n-k-j)$ and $\phi'_{2,j}=\alpha+\beta+k-\phi'_{1,j}$ for $1 \leq j \leq n-k$.  The mapping $\tau$ allows us to convert partitions of 
		$Td_{\alpha,\beta}[n-1,n-k]$ to partitions of $T_{\alpha,\beta}[k,n-k]$, and vice-versa.  This technique will prove useful in establishing identities involving the $\V$-Stirling 
		numbers.		

	\section{Recurrence Relations}
		The $\V$-Stirling numbers have the following initial values: $c^{\V}_{\alpha,\beta}[0,k]=S^{\V}_{\alpha,\beta}[0,k]=\delta_{0,k}$, 
		$S^\V_{\alpha,\beta}[n,0]=(v_{\alpha}w_{\beta})^n$, and $c^\V_{\alpha,\beta}[n,0]=v_{\alpha+n-1}w_{\beta}v_{\alpha+n-2}w_{\beta+1}\cdots 
		v_{\alpha}w_{\beta+n-1}$.  These initial values together with the following recurrence relations enable us to compute for the values of the $\V$-Stirling numbers.
		\begin{thm} 
			The $\V$-Stirling numbers satisfy the following recurrence relations:
			\begin{enumerate}[\emph{(}a\emph{)}]
				\item Triangular Recurrence Relations
				\begin{align}
					c^{\V}_{\alpha ,\beta}[n,k]&=c^\V_{\alpha ,\beta+1}[n-1,k-1] + (v_{\alpha+n-1} w_{\beta})c^\V_{\alpha ,\beta+1}[n-1,k] \label{1trra}\\
					S^{\V}_{\alpha ,\beta}[n,k]&=S^\V_{\alpha ,\beta+1}[n-1,k-1] + (v_{\alpha+k}w_{\beta}) S^\V_{\alpha ,\beta}[n-1,k] \label{2trra}
				\end{align}
				
				\item Vertical Recurrence Relations
				\begin{align}
					c^\V_{\alpha,\beta}[n+1,k+1]&=\sum_{j=k}^{n} (v_{\alpha+n} w_{\beta} v_{\alpha+n-1} w_{\beta+1} \cdots v_{\alpha+j+1} w_{\beta+n-j-1})  
					c^\V_{\alpha,\beta+n-j+1}[j,k] \label{1vrra} \\
					S^\V_{\alpha,\beta}[n+1,k+1]&=\sum_{j=k}^{n} \left(v_{\alpha+k+1} w_{\beta}\right)^{n-j} S^\V_{\alpha,\beta+ 1}[j,k] \label{2vrra}
				\end{align}
				
				\item Horizontal Recurrence Relations
				\begin{align}
					c^\V_{\alpha,\beta}[n,k]&=\sum_{j=k}^{n} (-1)^{j-k} (v_{\alpha+n}w_{\beta-1})^{j-k}  c^\V_{\alpha,\beta-1}[n+1,j+1] \label{1hrra} \\
					S^\V_{\alpha,\beta}[n,k]&=\sum_{j=0}^{n-k} (-1)^j (v_{\alpha+k+1}w_{\beta-1}v_{\alpha+k+2}w_{\beta-2} \cdots v_{\alpha+k+j}w_{\beta-j}) 
					S^\V_{\alpha,\beta-j-1}[n+1,k+j+1] \label{2hrra}
				\end{align}
			\end{enumerate}
			\label{rec}
		\end{thm}
		\begin{proof}
			The triangular and vertical recurrence relations may be derived by a suitable partition of the corresponding sets of tableaux. For instance, \eqref{1trra} follows from 
			\eqref{cvwt} and the fact that
			\[Td_{\alpha,\beta}[n-1,n-k]=Td_{\alpha,\beta+1}[n-2,n-k] \cup ({\{\rho\}} \boxtimes Td_{\alpha,\beta+1}[n-2,n-k-1])\]
			where $\rho=\big[\begin{smallmatrix}\alpha+n-1\\\beta\end{smallmatrix}\big]$.  
			
			Denote by $\rho_j$ the $2\times j$ array whose first row entries are all $\alpha+n$ and second row entries are all $\beta-1$.  To obtain \eqref{1hrra}, assign the weight 
			$(-1)^j\V_{\wt}(\phi)$ to every tableau in the multiset $\cup_{j=k}^n \{\rho_j\} \boxtimes Td_{\alpha,\beta-1}[n,n-j]$, get the sum of all such weights, and apply \eqref{cvwt}.
			The proof of \eqref{2hrra} is analogous.
		\end{proof}
		
		Since the ring $K$ is commutative, if $\V=(v,w)$ and $\V'=(w,v)$ then $c_{\alpha,\beta}^{\V}[n,k]=c_{\beta,\alpha}^{\V'}[n,k]$ and 
		$S_{\alpha,\beta}^{\V}[n,k]=S_{\beta,\alpha}^{\V'}[n,k]$.  Thus, every identity in Theorem~\ref{rec} is equivalent to another identity.  As an example, we have the other 
		horizontal recurrence relation
		\begin{equation}
			c^\V_{\alpha,\beta}[n,k]=\sum_{j=k}^{n}(-1)^{j-k}(v_{\alpha-1}w_{\beta+n})^{j-k}c^\V_{\alpha-1,\beta}[n+1,j+1]\label{tobe}\,.
		\end{equation}
		This reveals a ``duality'' between identities that may look different but are actually equivalent, which occurs in cases where the weight functions are completely different. 
		To illustrate this, we consider the classical Stirling numbers that are obtained when $\V=(i,1)$ and $\alpha=\beta=0$. Here, \eqref{1hrra} reduces to
		\[c(n,k)=\sum_{j=k}^{n}(-n)^{j-k}c(n+1,j+1)\,.\]
		Denote by $c_r(n,k)$ and $S_r(n,k)$ the non-central Stirling numbers with parameter $r$ \cite{Kou, Mer}.  The non-central Stirling numbers are $\V$-Stirling numbers with
		$\V=(i,1)$, $\alpha=r$, and $\beta=0$. Equation~\eqref{tobe} then becomes
		\[c(n,k)=\sum_{j=k}^{n}c_{-1}(n+1,j+1)\,.\]
		Using Carlitz's identity~\cite[Theorem 2.2]{Med}, we get the double-sum identity
		\[c(n,k)=\sum_{j=k}^n\sum_{t=j+1}^{n+1}(-1)^{t-j-1}\binom{t}{j+1}c(n+1,t)\,.\]
		We will not mention anymore these alternative equivalent identities in the succeeding results.
	
	\section{Generating Functions}
		We introduce the following notation. For $n>0$,
		\[[x]^{(n)}_{\alpha,\beta}:=(x-v_{\alpha+n-1}w_{\beta-n+1})(x-v_{\alpha+n-2}w_{\beta-n+2})\cdots(x-v_{\alpha}w_{\beta})\,,~~[x]^{(0)}_{\alpha,\beta}:=1.\]
		\begin{thm} 
			The $\V$-Stirling numbers satisfy the following generating functions:
			\begin{align}
				\sum_{k=0}^n c^\V_{\alpha,\beta}[n,k]x^k&=(x+v_{\alpha+n-1}w_{\beta})(x+v_{\alpha+n-2}w_{\beta+1})\cdots (x+v_{\alpha}w_{\beta+n-1})\label{cgf} \\
				\sum_{n \geq k} S^\V_{\alpha,\beta}[n,k] x^n &= 
				\frac{x^k}{(1-xv_{\alpha+k}w_{\beta})(1-xv_{\alpha+k-1}w_{\beta+1})\cdots(1-xv_{\alpha}w_{\beta+k})}\label{Sgf}\\
				x^n&=\sum_{k=0}^n S^\V_{\alpha,\beta-k}[n,k] [x]^{(k)}_{\alpha,\beta}\label{weirdgf}
			\end{align}
			\label{genfcn}
		\end{thm}
		\begin{proof}
			The generating functions \eqref{cgf} and \eqref{Sgf} follow directly from the definition of the $\V$-Stirling numbers. On the other hand, we prove
			\eqref{weirdgf} by induction. Assume that \eqref{weirdgf} holds for some $n\in\mathbb N$ such that $n>0$.  Using \eqref{2trra},
			\begin{align*}
				\sum_{k=0}^{n+1}S^\V_{\alpha,\beta-k}[n+1,k][x]^{(k)}_{\alpha,\beta}
				&= \sum_{k=0}^{n} S^\V_{\alpha,\beta-(k+1)+1}[n,k] [x]^{(k+1)}_{\alpha,\beta} + \sum_{k=0}^{n} v_{\alpha+k}w_{\beta-k} S^\V_{\alpha,\beta-k}[n,k] 
				[x]^{(k)}_{\alpha,\beta}\\
				&= \sum_{k=0}^{n} S^\V_{\alpha,\beta-k}[n,k] (x-v_{\alpha+k}w_{\beta-k})[x]^{(k)}_{\alpha,\beta}+ \sum_{k=0}^{n} v_{\alpha+k}w_{\beta-k}
				S^\V_{\alpha,\beta-k}[n,k] [x]^{(k)}_{\alpha,\beta}\\
				&= \sum_{k=0}^{n}  S^\V_{\alpha,\beta-k}[n,k] \left(x[x]^{(k)}_{\alpha,\beta}\right)\\
				&= x^{n+1}\,.
				\end{align*}
		\end{proof}

		Let $\V=(p^i,q^i)$ and $\alpha=\beta=0$. From \eqref{cgf} we obtain
		\[\sum_{k=0}^n (pq)^{\binom{n-k}{2}}\pq{n}{k}_{p,q} x^k = (x+p^{n-1})(x+p^{n-2}q) \cdots (x+q^{n-1})\,.\]
		Replacing $p$ with $p/q$ and $q$ by 1, and then multiplying both sides by $q^{\binom{n}{2}}$, we get
		\[\sum_{k=0}^n p^{\binom{n-k}{2}} q^{\binom{k}{2}} \pq{n}{k}_{p,q} x^k=(p^{n-1}+xq^{n-1})(p^{n-2}+xq^{n-2}) \cdots (1+x)\,.\]
		which is exactly Theorem 3 in \cite{Cor}.

		On the other hand, \eqref{Sgf} reduces to
		\[\sum_{n\geq k} \pq{n}{k}_{p,q} x^n = \frac{x^k}{(1-xp^k)(1-xp^{k-1}q)\cdots(1-xq^{k})}\,.\]
		From \eqref{weirdgf}, we obtain
		\[q^{\binom{n}{2}}x^n=\sum_{k=0}^n (-1)^{n-k} q^{\binom{n-k}{2}}\pq{n}{k}_{p,q} (xq^{k-1}+p^{k-1}) (xq^{k-2}+p^{k-2}) \cdots (x+1)\,,\]
		which is found in Theorem 5 of \cite{Cor}.

	\section{Orthogonality and Inverse Relations}
		The binomial coefficients and Stirling numbers satisfy the orthogonality relations
		\[\sum_{k=m}^n (-1)^{n-k} \binom{n}{k}\binom{k}{m} = \delta_{n,m}\quad\text{and}\quad \sum_{k=m}^n (-1)^{n-k} c(n,k)S(k,m) = \delta_{n,m}.\]
		The corresponding relations for the $\V$-Stirling numbers are given in the next theorem.

		\begin{thm}
			Let $m\leq n$. Then the following orthogonality relations hold:
			\begin{align}
				\sum_{k=m}^{n} (-1)^{n-k} c^\V_{\alpha,\beta+m+1}[n,k] S^\V_{\alpha,\beta+n}[k,m]&=\delta_{n,m}\label{orthorth}\\
				\sum_{k=m}^{n} S^\V_{\alpha,\beta}[n,k] (-1)^{k-m} c^\V_{\alpha,\beta+1}[k,m]&=\delta_{n,m}\label{betterorth}
			\end{align}
			Equivalently, the following pairs of matrices are inverses of each other:
			\begin{align}
				\left\langle (-1)^{n-k} c^\V_{\alpha,\beta-n+1}[n,k] \right\rangle ~&,~\left\langle S^\V_{\alpha,\beta-k}[n,k] \right\rangle\;\text{and} \label{refmat}\\
				\left\langle (-1)^{n-k} c^\V_{\alpha,\beta+1}[n,k] \right\rangle ~&,~ \left\langle S^\V_{\alpha,\beta}[n,k] \right\rangle\label{refmatt}
			\end{align}
		\end{thm}
		\begin{proof}
			The generating functions \eqref{cgf} and \eqref{weirdgf}, where in the former we replaced $x$ with $-x$ and $\beta$ with $\beta-n+1$, imply that
			$\left\langle (-1)^{n-k} c^\V_{\alpha,\beta-n+1}[n,k] \right\rangle$ and $\left\langle S^\V_{\alpha,\beta-k}[n,k] \right\rangle$ are change of basis matrices
			between the bases $\{1,x,\ldots,x^n\}$ and $\{1,[x]_{\alpha,\beta}^{(1)},\ldots,[x]_{\alpha,\beta}^{(n)}\}$.  This proves \eqref{orthorth} and \eqref{refmat}.  
			The proof of \eqref{betterorth} and \eqref{refmatt} is similar.
		\end{proof}
	
		When $\V=(p^i,q^i), \alpha=\beta=0$, the orthogonality relation \eqref{betterorth} becomes
		\[\sum_{k=m}^n (-1)^{k-m} p^{\binom{k-m}{2}} q^{\binom{n-k}{2}} \pq{n}{k}_{p,q}   \pq{k}{m}_{p,q} =\delta_{n,m}\,.\]
		For an alternative derivation, see Theorem 4 of \cite{Cor}.

		Given any two sequences $(a_n)_{n\geq 0}$ and $(b_n)_{n\geq 0}$ from a commutative ring $K$ with unity, the orthogonality relations give us the following 
		inverse relations
		\begin{alignat*}{3}
			a_n &=\sum_{k=0}^n (-1)^{n-k}c^\V_{\alpha,\beta-n+1}[n,k] b_k\quad&&\Longleftrightarrow\quad & b_n &=\sum_{k=0}^n S^\V_{\alpha,\beta-k}[n,k] a_k\\ 
			a_n &=\sum_{k=0}^n S^\V_{\alpha,\beta}[n,k] b_k&&\Longleftrightarrow   & b_n &=\sum_{k=0}^n (-1)^{n-k} c^\V_{\alpha,\beta+1}[n,k] a_k
		\end{alignat*}
		Variants of these relations also exist. For instance, it is easy to verify that
		\[\sum_{k=m}^{n} (-1)^{n-k} c^\V_{\alpha,\beta+m+1}[n+\gamma,k+\gamma] S^\V_{\alpha,\beta+n}[k+\gamma,m+\gamma]=\delta_{n,m}\,.\]
		In addition, if $r\geq n$, then by taking the transpose of the matrices in \eqref{refmat}, we obtain
		\[a_n =\sum_{k=0}^r (-1)^{n-k}c^\V_{\alpha,\beta-k+1}[k,n] b_k\quad\Longleftrightarrow \quad b_n =\sum_{k=0}^r S^\V_{\alpha,\beta-n}[k,n] a_k\,.\]

	\section{A Factorization of $\V$-Stirling Matrices}
		One of the widely-used combinatorial identities is the Vandermonde convolution given by
		\[\binom{m_1+m_2}{n} = \sum_{k=0}^n \binom{m_1}{n-k}\binom{m_2}{k}\,.\]
		By letting $m_1=j$, $m_2=s+i$, and $n=s+j$, we obtain the alternative form
		\[\binom{s+i+j}{s+j} = \sum_{k=0}^{r} \binom{s+i}{s+k} \binom{j}{j-k}\,.\]
		for some $r$.  This yields the matrix LU-decomposition
		\[\left\langle \binom{s+i+j}{s+j}\right\rangle=\left\langle \binom{s+i}{s+j}\right\rangle\left\langle \binom{j}{j-i}\right\rangle\,.\]
		
		We will obtain an analogue of these for the $\V$-Stirling numbers.  First, we derive the corresponding convolution identity.
		
		\begin{thm}\label{conv}
			Let $m_1,m_2$, and $n$ be integers. Then
			\begin{align}
			c^\V_{\alpha,\beta}[m_1+m_2,n]&=\sum_{k=0}^{n} c^\V_{\alpha+m_2,\beta}[m_1,n-k] c^\V_{\alpha,\beta+m_1}[m_2,k] \label{ccon}\\
			S^\V_{\alpha,\beta}[m_1+m_2,n]&=\sum_{k=0}^{n} S^\V_{\alpha+k,\beta}[m_1,n-k] S^\V_{\alpha,\beta+n-k}[m_2,k]\,. \label{Scon}
			\end{align}
		\end{thm}
		\begin{proof}
			Any $\phi\in Td_{\alpha,\beta}[m_1+m_2-1,m_1+m_2-n]$ may be written as $\phi=\phi_1\boxtimes\phi_2$, where $\phi_1\in Td_{\alpha+m_2,\beta}[m_1-1,m_1-n+k]$ and
			$\phi_2\in Td_{\alpha,\beta+m_1}[m_2-1,m_2-k]$ for some unique $k$ where $0\leq k\leq n$. This enables us to write
			\begin{equation}
				Td_{\alpha,\beta}[m_1+m_2-1,m_1+m_2-n]=\bigcup_{k=0}^n Td_{\alpha+m_2,\beta}[m_1-1,m_1-n+k] \boxtimes Td_{\alpha,\beta+m_1}[m_2-1,m_2-k]. \label{hereitis}
			\end{equation}
			Note that the union above is disjoint by the uniqueness of $k$.

			For a fixed $\phi\in Td_{\alpha,\beta}[m_1+m_2-1,m_1+m_2-n]$, pick the least possible value of $k$, where $0\leq k \leq n$, such that the first row entries
			of the last $m_2-k$ columns of $\phi$ are from ${\{\alpha,\alpha+1,\ldots,\alpha+m_2-1\}}$. This means that the first row entries of the remaining $m_1-n+k$
			columns belong to \[{\{\alpha+m_2,\ldots,\alpha+m_2+m_1-1\}}.\]			
			Hence we have
			\begin{align*}
				c^\V_{\alpha,\beta}[m_1+m_2,n]&=\textstyle\sum_{\phi\in Td_{\alpha,\beta}[m_1+m_2-1,m_1+m_2-n]}\V_{\wt}(\phi)\\
				&=\sum_{k=0}^n \left(\textstyle\sum_{\phi_1\in Td_{\alpha+m_2,\beta}[m_1-1,m_1-n+k]} \V_{\wt}(\phi_1) \right)\left(\textstyle\sum_{\phi_2\in 
				Td_{\alpha,\beta+m_1}[m_2-1,m_2-k]}  \V_{\wt}(\phi_2) \right) \\
				&=\sum_{k=0}^{n} c^\V_{\alpha+m_2,\beta}[m_1,n-k]\,c^\V_{\alpha,\beta+m_1}[m_2,k].
			\end{align*}
			This proves \eqref{ccon}.

			To prove \eqref{Scon}, we find a partition of $T_{\alpha,\beta}[n,m_1+m_2-n]$ similar to \eqref{hereitis}. Applying
			$\tau$ to both sides of \eqref{hereitis} yields
			\[T_{\alpha,\beta}[n,m_1+m_2-n]=T_{\alpha+k,\beta}[n-k,m_1-n+k]\boxtimes T_{\alpha,\beta+n-k}[k,m_2-k].\]
			The proof of \eqref{Scon} then proceeds analogously to that of \eqref{ccon}.
		\end{proof}

		Another variation to the binomial convolution was given by Gould and Srivastava in \cite{Gou} and reads
		\[\binom{m_1+m_2}{r+s}=\sum_{k=-\infty}^{+\infty} \binom{m_1}{r-k} \binom{m_2}{s+k}.\]
		Setting $n=r+s$, we obtain the corresponding variation to Theorem \ref{conv}:
		\begin{align*}
			c^\V_{\alpha,\beta}[m_1+m_2,r+s]&=\sum_{k=-\infty}^{+\infty} c^\V_{\alpha+m_2,\beta}[m_1,r-k]c^\V_{\alpha,\beta+m_1}[m_2,s+k]\\
			S^\V_{\alpha,\beta}[m_1+m_2,r+s]&=\sum_{k=-\infty}^{+\infty} S^\V_{\alpha+s+k,\beta}[m_1,r-k]S^\V_{\alpha,\beta+r-k}[m_2,s+k]
		\end{align*}
		
		The following theorem gives an LU-factorization of matrices whose entries are $\V$-Stirling numbers.

		\begin{thm} For $s\in\mathbb{N}$,\label{erhhh}
			\begin{align}
				\left\langle c^\V_{\alpha-i,\beta-j}[s+i+j,s+j]\right\rangle&=\left\langle c^\V_{\alpha-i,\beta}[s+i,s+j]\right\rangle
				\left\langle c^\V_{\alpha+s,\beta-j}[j,j-i]\right\rangle \label{reffac} \\
				\left\langle S^\V_{\alpha,\beta-j}[s+i+j,s+j]\right\rangle&= \left\langle S^\V_{\alpha,\beta-j}[s+i,s+j]\right\rangle
				\left\langle S^\V_{\alpha+s+i,\beta-j}[j,j-i] \right\rangle\,.\label{secdet}
			\end{align}
		\end{thm}
		\begin{proof} 
			Using the convolution identity in \eqref{ccon} with $m_1=j$, $m_2=s+i$, and $n=s+j$, and $\alpha$ replaced with $\alpha-i$ and $\beta$ replaced with
			$\beta-j$, we have
			\begin{align*}
				c^\V_{\alpha-i,\beta-j}[s+i+j,s+j]
				&=\sum_{k=0}^{s+j} c^\V_{\alpha+s,\beta-j}[j,s+j-k] c^\V_{\alpha-i,\beta}[s+i,k] \\
				&=\sum_{k=-s}^{j} c^\V_{\alpha+s,\beta-j}[j,j-k] c^\V_{\alpha-i,\beta}[s+i,s+k] \\
				&=\sum_{k=0}^{r } c^\V_{\alpha-i,\beta}[s+i,s+k] c^\V_{\alpha+s,\beta-j}[j,j-k]\,,
			\end{align*}
			for some $r\in\mathbb{N}$, $r\geq i,j$.		
			
			Similarly,
			\[S^\V_{\alpha,\beta-j}[s+i+j,s+j]=\sum_{k=0}^{r} S^\V_{\alpha,\beta-k}[s+i,s+k] S^\V_{\alpha+s+k,\beta-j}[j,j-k]\,,\]
			which proves \eqref{secdet}.
		\end{proof}

		Note that $\langle c^{\V}_{\alpha-i,\beta}[s+i,s+j]\rangle$ and $\langle S^{\V}_{\alpha,\beta-j}[s+i,s+j]\rangle$ are lower-triangular
		matrices while $\langle c^{\V}_{\alpha+s,\beta-j}[j,j-i]\rangle$ and $\langle S^{\V}_{\alpha+s+i,\beta-j}[j,j-i]\rangle$ are
		upper-triangular matrices.  We then obtain the following corollary.
		
		\begin{cor} For $s\in\mathbb{N}$,
			\begin{align}
				\Det\left( \left\langle c^\V_{\alpha-i,\beta-j}[s+i+j,s+j]\right\rangle_{0\leq i,j \leq r}\right)&=\prod_{k=0}^r v_{\alpha+s+k-1}w_{\beta-k} 
				v_{\alpha+s+k-2}w_{\beta-k+1} \cdots v_{\alpha+s}w_{\beta-1}	\label{detc}\\
				\Det\left( \left\langle S^\V_{\alpha,\beta-j}[s+i+j,s+j]\right\rangle_{0\leq i,j \leq r}\right)
				&=\prod_{k=0}^r{\left(v_{\alpha+s+k}w_{\beta-k}\right)}^k\,. 
				\label{detS}
			\end{align}
		\end{cor}

		A $q$-analogue of the Stirling numbers studied by Ehrenborg \cite{Ehr} is given by the recurrence relation
		\[\hat{S}_q[n,k]=q^{k-1} S_q[n-1,k-1] + [k]_q S_q[n-1,k].\]
		On the other hand, the $q$-analogue resulting from the $\V$-Stirling numbers obtained by letting $\alpha=\beta=0$ and $\V=([i]_q,1)$ satisfies the recurrence
		relation
		\[S^\V_{0,0}[n,k] = S_q[n,k] = S_q[n-1,k-1] + [k]_q S_q[n-1,k]\,.\]
		where $[i]_q=\frac{1-q^i}{1-q}$. These analogues are related by $\hat{S}_q[n,k] = q^{\binom{k}{2}} S_q[n,k]$ (see \cite{Cig}).

		Note that
		\begin{align*}
			\langle \hat{S}_q[s+i+j,s+j]\rangle_{0\leq i,j\leq r}&= \langle q^{\binom{s+j}{2}} S_q[s+i+j,s+j]\rangle\\
			&= \left\langle S_q[s+i+j,s+j]\right\rangle\mbox{diag}\left(q^{\binom{s}{2}}, q^{\binom{s+1}{2}}, \cdots, q^{\binom{s+r}{2}}\right)\,.
		\end{align*}
		By \eqref{detS}, $\Det\left(\left\langle S_q[s+i+j,s+j]\right\rangle\right)={[s]_q^0} {[s+1]_q^1} \cdots {[s+r]_q^r}$. In addition, the determinant of
		\[\mbox{diag}\left(q^{\binom{s}{2}}, q^{\binom{s+1}{2}}, \ldots, q^{\binom{s+r}{2}}\right)\] is given by $q^{\binom{s+r+1}{3}-\binom{s}{3}}$ using Chu 
		Shih-Chieh's identity.  Thus,
		\[\Det\left(\langle\hat{S}_q[s+i+j,s+j]\rangle\right)=q^{\binom{s+r+1}{3}-\binom{s}{3}}[s]_q^0[s+1]_q^1\cdots[s+r]_q^r.\]
		The same conclusion was reached in \cite{Ehr} via a combinatorial argument.

	\section{Combinatorial Interpretations}

		Let $\mathbb N[i]$ denote the set of polynomials in $i$ with coefficients from $\mathbb{N}$. Throughout this section, we will assume that $\V=(v,w)$, where
		$v,w\in\mathbb N[i]$, unless stated otherwise. We will derive a general method of obtaining combinatorial interpretations for the $\V$-Stirling numbers
		in terms of certain colored partitions and permutations. Another approach can be found in \cite{Med}, which may also be applied to some of the particular cases
		considered here.

		Consider a $\B$-tableau $\phi$. A $0,1$-tableau of shape $\phi$ is an array of top- and left-justified boxes such that the length of the $j$th column is $\phi_{1,j}$ and 
		exactly one box in each column is filled with $1$ while the rest are filled with $0$'s. If $\V=(i,1)$, then the number of $0,1$-tableaux of shape $\phi$ is $\V_{\wt}(\phi)$. 
		Note that we require each column to be of positive length.  Hence, if $\phi$ contains a first row entry equal to $0$, then the number of $0,1$-tableaux of shape $\phi$ is $0$.

		We now introduce a generalization of the $0,1$-tableau which we shall call the \emph{$0,1_{\V}$-tableau}.  Specifically, for a $\B$-tableau $\phi\in T_{\alpha,\beta}[r,s]$, we
		define a $0,1_{\V}$-tableau of shape $\phi$ as an $(\alpha+\beta+r+2)\times s$ rectangular array of boxes partitioned by a lattice path consisting of vertical (up to down) and
		horizontal (right to left) steps from the upper right hand to the lower left hand corner, satisfying the following properties:
		\begin{enumerate}[1.]
			\item The length (i.e., the number of boxes) on the $j$th column above (respectively, below) the lattice path is $\phi_{1,j}+1$ (respectively,
			$\phi_{2,j}+1$).
		
			\item Exactly two boxes in each column is filled with $1$; one is above the lattice path and the other below the lattice path. The rest of the boxes are filled with zeros.
		
			\item If $1$ is placed on the first (respectively, last) row, then $1$ is assigned one out of $v(0)$ (respectively, $w(0)$) colors.  If $1$ is placed on other rows in the
			$j$th column, then it is assigned one out of $[v(\phi_{1,j})-v(0)]/\phi_{1,j}$ colors if it is above the lattice path, and one out of $[w(\phi_{2,j})-w(0)]/\phi_{2,j}$
			colors if it is below the lattice path.  If $v(0)=0$ or $w(0)=0$, then we cannot place a $1$ on the first and last rows, respectively.
		\end{enumerate}

		Observe that the number of $0,1_{\V}$-tableaux of shape $\phi$ is $\V_{\wt}(\phi)$. When $\V=(i,1)$, we may ignore the first row of boxes since they cannot contain a 1, as well as
		the boxes below the path since only the last row can contain a 1. That is, we get the usual $0,1$-tableau in this case.

		We now denote by $T^{0,1_{\V}}_{\alpha,\beta}[k,n-k]$ the set of $0,1_{\V}$-tableaux of shape $\phi$ where $\phi\in T_{\alpha,\beta}[k,n-k]$. The set
		$Td^{0,1_{\V}}_{\alpha,\beta}[n-1,n-k]$ is defined analogously. Let $[n]_0=[n]\cup\{0\}$. We consider first $0,1_{\V}$-tableaux where $\V=(v,1)$.  We define $\Part(n,k;v)$
		to be the set of partitions of $[n]_0$ into $(k+1)$ blocks $B_0, B_1, \ldots, B_k$ such that: 
		\begin{enumerate}[(a)]
			\item All block minima are not colored. 
			
			\item If $a_1<a_2<\cdots<a_{n-k}$ are the elements of $[n]_0$ which are not subset minima, then $a_j$ takes one out of $v(0)$ colors if $a_j\in B_0$, and one out of 
			$[v(a_j-j)-v(0)]/(a_j-j)$ colors otherwise.
		\end{enumerate}
		Also, we define the set of colored permutations $\Perm(n,k;v)$ as the set of permutations of $[n]_0$ into $(k+1)$ disjoint cycles $C_0, C_1,\ldots, C_k$ such that:
		\begin{enumerate}[(a)]
			\item All cycle minima are not colored. 
			
			\item If $a_1<a_2<\cdots<a_{n-k}$ are the elements of $[n]_0$ which are not cycle minima, then $a_j$ takes one out of $v(0)$ colors if $a_j\in C_0$, and one out of 
			$[v(a_j-1)-v(0)]/(a_j-1)$ colors otherwise.
		\end{enumerate}
		
		There exists a bijection between the set $T^{0,1_{\V}}_{\alpha,\beta}[k,n-k]$ and ${\Part(n+\alpha+\beta,k+\alpha+\beta;v)}$, and between $Td^{0,1_{\V}}_{\alpha,\beta}[n-1,n-k]$
		and $\Perm(n+\alpha+\beta,k+\alpha+\beta;v)$, with the added restriction that $1,2,\ldots,\alpha$ are in distinct subsets or cycles, respectively. For a 
		$0,1_{\V}$-tableau $\phi_{0,1_{\V}}$ of shape $\phi \in T_{\alpha,\beta}[k,n-k]$, we obtain the corresponding partition in $\Part(n+\alpha+\beta,k+\alpha+\beta;v)$ as follows.
		Label the steps (except the last one, which we will ignore) by $0,1,\ldots,n+\alpha+\beta$. If the $i$th step on the lattice path is a vertical step, then $i$ is a block minimum 
		of the associated partition. On the other hand, if the $i$th step is a horizontal step which is in a column with a $1$ on the $r$th box from the top, then we put $i$ in the 
		$r$th block and we assign to $i$ the same color as that of $1$. For the other bijection, given a $0,1_{\V}$-tableau $\phi_{0,1_{\V}}$ of shape $\phi\in Td_{\alpha,\beta}[n-1,n-k]$,
		remove a vertical step after every horizontal step in the lattice path on $\phi_{0,1_{\V}}$.  Label the steps in the resulting (disconnected) lattice path by 
		$0,1,\ldots,n+\alpha+\beta$.  If the $i$th step is vertical, then $i$ is a cycle minimum.  Consider the word $\omega_0=s_0s_1\cdots s_{\alpha+\beta+k}$, where the $s_i$'s are the
		cycle minima arranged in increasing order.  For each $1\leq j\leq n-k$, suppose that the $1$ above the $j$th horizontal step is in the $r_j$th box from the top.  Form the word 
		$\omega_j$ by inserting the label $i_j$ of the $j$th horizontal step after the $r_j$th letter of $\omega_{j-1}$, with the color	of $1$ assigned to $i_j$.  The permutation
		induced by the word $w_{n-k}$ is the corresponding permutation in $\Perm(n+\alpha+\beta,k+\alpha+\beta;v)$.  Note that when $v(0)=0$, the zeroth cycle or block consist only of 
		$0$, and thus, may be ignored.

		\begin{figure}[ht]
			\begin{minipage}[t]{72pt}
				\centering
				$\phi_{0,1_{\V}}$\\
				\psset{xunit=18pt,yunit=18pt}
				\begin{pspicture}(0,0)(3,7)	
					\multido{\n=0+1}{4}{\psline[linewidth=0.5pt](\n,0)(\n,7)}
					\multido{\n=0+1}{8}{\psline[linewidth=0.5pt](0,\n)(3,\n)}
					\psline[linewidth=1.75pt](0,0)(0,3)(1,3)(1,5)(3,5)(3,7)
					\rput(0.5,0.5){$1$}
					\rput(1.5,0.5){$1$}
					\rput(2.5,0.5){$1$}
					\rput(0.5,1.5){$0$}
					\rput(1.5,1.5){$0$}
					\rput(2.5,1.5){$0$}
					\rput(0.5,2.5){$0$}
					\rput(1.5,2.5){$0$}
					\rput(2.5,2.5){$0$}
					\rput(0.5,3.5){$0$}
					\rput(1.5,3.5){$0$}
					\rput(2.5,3.5){$0$}
					\rput(0.5,4.5){$1_2$}
					\rput(1.5,4.5){$0$}
					\rput(2.5,4.5){$0$}
					\rput(0.5,5.5){$0$}
					\rput(1.5,5.5){$0$}
					\rput(2.5,5.5){$1_1$}
					\rput(0.5,6.5){$0$}
					\rput(1.5,6.5){$1_3$}
					\rput(2.5,6.5){$0$}
				\end{pspicture}
			\end{minipage}
			\begin{minipage}[t]{72pt}
				\centering
				$\psi_{0,1_{\V}}$\\
				\psset{xunit=18pt,yunit=18pt}
				\begin{pspicture}(0,0)(3,7)	
					\multido{\n=0+1}{4}{\psline[linewidth=0.5pt](\n,0)(\n,7)}
					\multido{\n=0+1}{8}{\psline[linewidth=0.5pt](0,\n)(3,\n)}
					\psline[linewidth=1.75pt](0,0)(0,3)(1,3)(1,5)(2,5)(2,6)(3,6)(3,7)
					\rput(0.5,0.5){$1$}
					\rput(1.5,0.5){$1$}
					\rput(2.5,0.5){$1$}
					\rput(0.5,1.5){$0$}
					\rput(1.5,1.5){$0$}
					\rput(2.5,1.5){$0$}
					\rput(0.5,2.5){$0$}
					\rput(1.5,2.5){$0$}
					\rput(2.5,2.5){$0$}
					\rput(0.5,3.5){$1_1$}
					\rput(1.5,3.5){$0$}
					\rput(2.5,3.5){$0$}
					\rput(0.5,4.5){$0$}
					\rput(1.5,4.5){$0$}
					\rput(2.5,4.5){$0$}
					\rput(0.5,5.5){$0$}
					\rput(1.5,5.5){$1_1$}
					\rput(2.5,5.5){$0$}
					\rput(0.5,6.5){$0$}
					\rput(1.5,6.5){$0$}
					\rput(2.5,6.5){$1_4$}
				\end{pspicture}
			\end{minipage}
			\caption{Examples of $0,1_v$-tableaux}\label{tabs}
		\end{figure}

		As an illustration, let $\V=(2i+4,1)$, and $\phi_{0,1_{\V}}$ and $\psi_{0,1_{\V}}$ be the $0,1_v$-tableaux in Figure~\ref{tabs}, where $\phi_{0,1_{\V}}$ 
		is of shape $\begin{bmatrix}3 & 1 & 1 \\ 2 & 4 & 4 \end{bmatrix} \in T_{0,0}[5,3]$ and $\psi_{0,1_{\V}}$ is of shape 
		$\begin{bmatrix} 3 & 1 & 0 \\ 2 & 4 & 5\end{bmatrix}\in Td_{0,0}[5,3]$. The lattice path associated with $\phi_{0,1_{\V}}$ is $VVHHVVHVVV$.  Hence, the 
		associated block minima are $0,1,4,5,7,8$. The second step on the lattice path is an $H$, which is on a column with a $1_1$ on the second row. This means that $2_1$
		is an element of the second block. Continuing, we obtain that the corresponding partition is $\{0,3_3\}\{1,2_1\}\{4,6_2\}\{5\}\{7\}\{8\}\in\Part(8,5;v)$. For $\psi_{0,1_{\V}}$, 
		the resulting lattice path is $VHHVHVV$.  Thus, the cycle minima are $0,3,5,6$ and $\omega_0=0356$, $\omega_1=01_4356$, $\omega_2=01_42_1356$, 
		$\omega_3=01_42_134_156$.  Therefore, the corresponding permutation is $(0\,1_4\,2_1)(3\,4_1)(5)(6)\in\Perm(6,3;v)$.

		We now apply the approach described above to several special cases.  
		\begin{enumerate}[1.]
			\item Merris' $p$-Stirling numbers \cite{Mer} and Koutras' non-central Stirling numbers $c_p(n,k)$ and $S_p(n,k)$ ~\cite{Kou} are obtained by letting 
			$\V=(i+p,1)$ and $\alpha=\beta=0$, or by letting $\V=(i,1)$, $\alpha=p$, and $\beta=0$.  We can therefore interpret, for example, $c_p(n,k)$
			as the number of permutations of $[n]_0$ into $(k+1)$ disjoint cycles such that the nonminimal elements of the zeroth cycle take $p$ colors, or as the number of
			permutations of $[n+p]$ into $(k+p)$ disjoint cycles such that $1,2,\ldots,p$ are in distinct cycles. Alternatively, if $[m,n]:=\{m,m+1,\ldots,n\}$,
			then $c_p(n,k)$ is the number of permutations of $[-p+1,n]$ into $(k+p)$ disjoint cycles such that the nonpositive elements are in distinct cycles.
			
			\item We get Sun's $p$-Stirling numbers $c^p(n,k)$ and $S^p(n,k)$ \cite{Sun} when $\V=(i^p,1)$ and $\alpha=\beta=0$.  Observe that the number of $0,1_\V$-tableaux is
			the same as the number of $p$-tuples of $0,1_{\V'}$-tableaux of identical shape, where $\V'=(i,1)$. Hence, for instance, $c^p(n,k)$ gives the number of permutations of $[n]$ into
			$k$ disjoint cycles such that every nonminimal element $a_j$ is assigned $(a_j-1)^{p-1}$ colors, as well as the the number of $p$-tuples of permutations of $[n]$ having
			$k$ disjoint cycles such that the permutations have identical cycle minima. It is easy to see that the latter interpretation is equivalent to that of Sun's, namely,
			that $c^p(n,k)$ counts the number of $k$-matrix permutations of $M(n,p)$ (see \cite{Sun}).
			
			\item The Jacobi Stirling numbers $Jc(n,k)$ and $JS(n,k)$ \cite{Eve} are obtained by setting $\V=(i(i+z),1)$, where $z\in\mathbb R$ and $\alpha=\beta=0$.
			They reduce to the Legendre Stirling numbers $Lc(n,k)$ and $LS(n,k)$ when $z=1$. Further generalizations of these sequences
			were given in OEIS, where $\V=(\prod_{i=1}^j (i+a_i),1)$, where $a_i\in\mathbb{N}$ (see for example, sequences A071951, A089504, A090215, and A090217 \cite{Bar}). 
			
			Using the approach described above, if $\V=(\prod_{i=1}^j (i+a_i),1)$, then for a $\B$-tableau $\phi$, $\V_{\wt}(\phi)$ equals the number of
			$0,1_{\V}$-tableaux of shape $\phi$, which in turn, equals the number of $j$ tuples $(\phi_{0,1_{{\V}_1}},\phi_{0,1_{{\V}_2}},\ldots,\phi_{0,1_{{\V}_j}})$
			where $\phi_{0,1_{{\V}_i}}$ is a $0,1_{{\V}_i}$-tableaux where ${\V_i=(i+a_i,1)}$. Hence, $c^\V_{\alpha,\beta}[n,k]$ is the number of $j$ tuples
			$(\sigma_1,\sigma_2,\ldots,\sigma_j)$, where $\sigma_i$ is a permutation of $[-a_i+1,n]$ into $(k+a_i)$ disjoint cycles such that the nonpositive elements are in distinct cycles
			and the positive minimal elements of $\sigma_1,\sigma_2,\ldots,\sigma_j$ are identical.
			
			Let $\V=(i(i+1),1)$ and $\alpha=\beta=0$. Let $[\pm n]_0 :={\{0,1,-1,2,-2,\ldots,n,-n\}}$.  We say that a subset forming a partition of $[\pm n]_0$ contains both
			copies of a nonzero integer $j$ if both $j$ and $-j$ are in that subset. Define $\Part_{\pm}(n,k)$ as the set of partitions of $[\pm n]_0$ into $k+1$ subsets 
			$D_0,D_1,D_2,\ldots,D_k$ satisfying the following properties:
			\begin{enumerate}[(a)]
				\item $0\in D_0$ and $D_0$ is not allowed to contain both copies of the same nonzero integer.
				
				\item The subsets $D_1,D_2,\ldots,D_k$ are nonempty and indistinguishable. The subset minimum of every subset $D_j$, where $1\leq j\leq k$, is the
				positive integer $|m|$ such that $|m|\leq |m'|$ for all $m'\in D_j$. Every subset must contain both copies of its subset minimum but no other element
				occurs in two copies in that subset.
			\end{enumerate}
			Gelineau and Zang \cite{Gel} showed that $LS(n,k)=|\Part_{\pm}(n,k)|$. It is straightforward to construct a	bijection between our combinatorial interpretation and those of Zhang
			and Gelineau. 
						
			\item Miceli's poly-Stirling numbers \cite{Mic} is exactly the case when $\V=(v,1)$, where $v\in\mathbb N[i]$ and $\alpha=\beta=0$.  We have thus provided an alternative
			interpretation from those given in~\cite{Mic}.

			\item All of the examples so far have considered only $\V$-Stirling numbers where $\V=(v,1)$. However, we can also apply the same method of obtaining a permutation or
			partition to the boxes below the lattice path. Recall that the $b$-Stirling number of the first kind are defined by the generating function in \eqref{ooh} and are
			obtained by letting $\V=(i,i)$ and $\alpha=\beta=0$.  They can be combinatorially interpreted as the number of pairs of permutations $(\sigma_1, \sigma_2)$, where 
			$\sigma_1$ is a permutation of $[n]$ into disjoint cycles $C_1,C_2,\ldots, C_k$, and $\sigma_2$ is a permutation of $[n]$ into disjoint cycles $C'_1,C'_2,\ldots, C'_k$, such
			that for $1\leq j\leq k$, the sum of the minimal elements of $C_j$ and $C'_{k-j+1}$ is $n+1$. Here, the cycles $C'_j$ are obtained by labeling the lattice path starting
			instead from the lower left hand corner to the upper right hand corner.
		\end{enumerate}

		One can extend the approach we described here to any nonnegative integral weight function.  This involves expressing each $v(\phi_{1,j})$ and $w(\phi_{2,j})$ as a
		suitable sum of at most $\phi_{1,j}+1$ and $\phi_{2,j}+1$ nonnegative integers, respectively, and will entail a more general class of colored permutations and partitions. 
		Details are left to the reader.
		
	\section{Recommendations}
		The $\V$-Stirling numbers for which one of the weight functions is not a constant constitute a family of sequences that do not belong to the $\U$-Stirling numbers.  What we
		currently know about these sequences are limited to properties considered in this paper. Explicit formulas of these sequences are, to the authors' knowledge, still unknown. 
		It might be interesting to explore the corresponding $q$- and $p,q$-analogues of these numbers, their combinatorial interpretations in terms of graphs, rook placements, and vector 
		spaces, and their expressions in terms of finite differences.  The bimodality of the $p,q$-binomial coefficients was recently proved by Su and Wang in ~\cite{Su}. It would also be
		worthwhile to determine the modality of $\V$-Stirling numbers for particular weight functions.

	\section*{Acknowledgements}
		K.J.M.~Gonzales would like to thank the Department of Science and Technology (DOST) through the ASTHRD Program for financial support during his stay at the University of the 
		Philippines Diliman.

	\bibliographystyle{amsplain}

\end{document}